\documentclass[11pt]{article}

\usepackage{array}
\newcolumntype{P}[1]{>{\centering\arraybackslash}p{#1}}

\usepackage{booktabs} 
\usepackage{multirow}
\usepackage{makecell}

\usepackage{hyperref}

\usepackage{color,soul}
\usepackage{xcolor}
\usepackage{framed}
\colorlet{shadecolor}{yellow!20}

\usepackage{tikz}
\usepackage{verbatim}

\usepackage{amssymb}
\usepackage{import}
\usepackage{lineno}
\usepackage{mathtools,bbm}
\usepackage{graphicx,url}
\usepackage{graphics}

\usepackage{algorithm}
\usepackage{algpseudocode}

\usepackage{setspace,graphicx,epstopdf,amsmath,amsfonts,amssymb,amsthm,versionPO,bm}
\usepackage{marginnote,datetime,enumitem,rotating,fancyvrb,scalerel}
\usepackage{hyperref}
\usepackage[noabbrev, capitalise]{cleveref}

\excludeversion{notes}		
\includeversion{links}          

\iflinks{}{\hypersetup{draft=true}}

\ifnotes{%
\usepackage[margin=1in,paperwidth=10in,right=2.5in]{geometry}%
\usepackage[textwidth=1.4in,shadow,colorinlistoftodos]{todonotes}%
}{%
\usepackage[margin=1in]{geometry}%
\usepackage[disable]{todonotes}%
}

\makeatletter\let\chapter\@undefined\makeatother 

\usepackage{jfe}          

\newtheorem{proposition}{Proposition}

\newtheorem{lemma}{Lemma}

\newtheorem{assumption}{Assumption}

 \usepackage{natbib}
 \setcitestyle{authoryear,open={(},close={)}}

\usepackage{float}
\usepackage[caption = false]{subfig}
\usepackage{blindtext}

\begin{document}

\setlist{noitemsep}  
\onehalfspacing      

\title{\textbf{Perturbed utility stochastic traffic assignment}}

\author{
    Rui Yao\\EPFL \and
    Mogens Fosgerau\footnote{Corresponding author: \href{mailto:mogens.fosgerau@econ.ku.dk}{mogens.fosgerau@econ.ku.dk}.}\\University of Copenhagen \and
    Mads Paulsen\\Technical University of Denmark \and
    Thomas Kjær Rasmussen\\Technical University of Denmark
        }

\date{} 

\renewcommand{\thefootnote}{\fnsymbol{footnote}}
\singlespacing
\maketitle

\vspace{-.2in}
\begin{abstract}
    This paper develops a fast algorithm for computing the equilibrium assignment with the perturbed utility route choice (PURC) model. Without compromise, this allows the significant advantages of the PURC model to be used in large-scale applications. We formulate the PURC equilibrium assignment problem as a convex minimization problem and find a closed-form stochastic network loading expression that allows us to formulate the Lagrangian dual of the assignment problem as an unconstrained optimization problem. To solve this dual problem, we formulate a quasi-Newton accelerated gradient descent algorithm (qN-AGD*). Our numerical evidence shows that qN-AGD* clearly outperforms a conventional primal algorithm as well as a plain accelerated gradient descent algorithm. qN-AGD* is fast with a runtime that scales about linearly with the problem size, indicating that solving the perturbed utility assignment problem is feasible also with very large networks.
\end{abstract}


\medskip


\thispagestyle{empty}


\onehalfspacing
\setcounter{footnote}{0}
\renewcommand{\thefootnote}{\arabic{footnote}}
\setcounter{page}{1}

\section{Introduction}\label{sec:intro}
Traffic assignment deals with the problem of allocating travel demands between a set of origin-destination (OD) pairs onto a congestible transportation network under specific behavioral assumptions \citep{sheffi1985urban}. This problem is central to transportation network planning and analysis.

The perturbed utility route choice (PURC) model \citep{Fosgerau2021a,Fosgerau2022} predicts the demand of a traveler as the network flow vector that solves a certain convex optimization problem. The model has a number of features that are very attractive for applications. It generates realistic predictions of equilibrium network flow with substitution patterns induced directly by the network structure, while allowing a priori any physically possible route in the transportation network without the need for choice set generation, and while predicting zero flow in irrelevant parts of the network. Not least, the PURC model can be estimated by linear regression.

The demand of individual travelers depends on the travel times in the network. The network is congestible, which means that travel times, in turn, depend on the aggregate demand on each link in the network. To make predictions, we generally impose the condition that demand and travel times are in equilibrium, i.e. that individual demands are those that result from the network travel times, while travel times are those that result from the aggregate demand of travelers. The equilibrium traffic assignment problem is to determine this point~\citep{sheffi1985urban}.

A naive approach to solving this problem would alternate between computing network flows conditional on travel times and updating travel times. This would, however, be extremely time-consuming as the network flow of each traveler is the solution to an individual convex optimization problem with as many free variables as there are links in the network. This approach would be particularly impractical in large networks since the problem complexity increases drastically with the number of travelers and network size.
This paper develops a much faster approach for computing the PURC traffic equilibrium by exploiting properties of the model, thus making it relevant for large-scale applications.

The first contribution of this paper is to formulate the traffic assignment problem as a single large constrained convex optimization problem. We show that the equilibrium solution exists uniquely. This means we can be sure always to find an equilibrium while not having to worry about the potential existence of multiple equilibria.

Second, we exploit the structure of the PURC model to derive a closed-form expression for the link flows as functions of the link costs and dual parameters corresponding to the flow conservation constraints. This enables us to formulate the Lagrangian dual of the primal traffic assignment problem in closed form as an unconstrained convex optimization problem. This dual traffic assignment problem is much faster to solve than the primal traffic assignment problem since it is unconstrained and involves only the dual variables as decision variables. 

Third, we propose a fast solution algorithm. We include the equilibrium condition for travel times as an auxiliary fixed-point problem and solve the dual traffic assignment problem conditional on travel times using the accelerated gradient descent (AGD) algorithm \citep{nesterov1983method,beck2009fast}. To avoid oscillation at later iterations, we employ the AGD* scheme~\citep{chambolle2015convergence}, which turns out to gain us considerable speed. Furthermore, using a quasi-Newton method, we scale the gradient by the Hessian diagonal to avoid line search for step sizes, thereby gaining additional speed. Furthermore, at each iteration, we propose to update link travel times by a Newton step such that the equilibrium condition for travel times is satisfied at convergence. We refer to the resulting algorithm as qN-AGD*.

We test the performance of our qN-AGD* algorithm with a range of networks of various sizes and find very satisfactory runtimes. Importantly, we find that the runtime scales about linearly with the size of the problem, i.e., the number of nodes times the number of origin-destination pairs, suggesting that our algorithm will perform well with very large problems.

The paper is organized as follows. Section \ref{sec:literature} comprises a literature review of stochastic assignment models and algorithms and their relations to perturbed utility models. Section \ref{sec:PURC} presents the link-based perturbed utility route choice model and an analysis of its solution properties. The perturbed utility-based stochastic traffic assignment model is formulated and analyzed in Section \ref{sec:PURC_TAP}, along with a closed-form network loading and an unconstrained dual assignment problem. Section \ref{sec:solution_algo} presents an efficient algorithm based on quasi-Newton accelerated gradient descent (qN-AGD*) for the dual formulation. Section \ref{sec:results} reports the results of the numerical experiments. Section \ref{sec:conclusion} concludes.

\section{Literature review}\label{sec:literature}
\subsection{Perturbed utility models}
The perturbed utility model \citep{McFadden2012, Allen2019} in its general form, describes consumer choice as a vector $x$ that maximizes a concave function $v^{\top}x- F(x)$, where $v$ is a vector of utility indexes and $F$ is the convex perturbation function. The consumption vector $x$ is constrained to lie in some budget set $B$. This general framework can represent a wide range of behavioral models as special cases \citep{Allen2019}. The perturbed utility model becomes a discrete choice model when the budget set is defined as the probability simplex. In particular, any additive random utility discrete choice model \citep{McFadden1981} can be represented as a perturbed utility model, where the perturbation function $F$ is the convex conjugate \citep{Rockafellar1970} of the surplus function of the additive random utility model \citep{Hofbauer2002}. The multinomial logit model is a perturbed utility model in which the perturbation function is the Shannon entropy \citep{Sorensen2019}.

In this paper, we employ the perturbed utility route choice model \citep{Fosgerau2021a}, described in detail in Section \ref{sec:PURC}. This is a perturbed utility model in which the budget set for a traveler is the set of network flow vectors that satisfy flow conservation of one unit demand traveling from the origin to the destination. Moreover, the convex perturbation function incorporates the network structure by being constructed as a sum of convex terms, one for each link in the network. As a consequence, correlation between alternatives is directly induced by the physical network structure.

\subsection{Stochastic traffic assignment models}\label{sec:lit_assignment_model}
A stochastic traffic assignment model predicts network flows in stochastic user equilibrium (SUE), under which no traveler can reduce their perceived travel cost by unilaterally changing their routing decision~\citep{daganzo1977stochastic}. A key building block for stochastic traffic assignment is the traveler route choice model that characterizes individual traveler behavior. There are two main challenges associated with route choice modeling: correlation between path utilities and choice set generation.

The generalized extreme value (GEV) model \citep{McFadden1978} applied to model route choice accounts for correlation in path utilities, as it allows a closed-form expression for capturing similarity between alternatives in the random error term~\citep{Prato2009}. Examples include the link-nested logit~\citep{Vovsha1998}, paired combinatorial logit~\citep{Chu1989}, and generalized nested logit~\citep{Wen2001}. Alternatively, in models such as C-logit~\citep{Cascetta1996} and path-size logit~\citep{Ben-Akiva1999,duncan2020path}, a correction term is added to the systematic utility function to capture the effect of overlapping paths. However, these models are typically path-based and therefore require explicit choice set generation, either ex-ante or by variants of column generation methods.

In general, the classical discrete choice models require a choice set to be defined. Route choice set generation is challenging in large networks where full path enumeration is not feasible. Deterministic generation methods construct a set of feasible paths with pre-defined rules like link penalty~\citep{de1993multidimensional}, and branch and bound~\citep{prato2006applying}. These deterministic methods are generally computationally efficient, but there is no guarantee they will reproduce the observed paths. Moreover, the restrictions involved in generating choice sets may cause bias in parameter estimation \citep{Frejinger2009}. Stochastic generation methods based on variants of random walk methods assume an underlying universal choice set~\citep{Frejinger2009,flotterod2013metropolis} can be used to avoid bias in parameter estimates but entails additional complications of path sampling. Under specific assumptions about the existence of a consideration set, recent developments in machine learning methods provide a data-driven approach to inferring the choice set generation process \citep{ton2018evaluating, yao2022variational}.

Link-based recursive route choice models avoid choice set generation by assuming travelers sequentially choose the next link at each node. This family of models includes the recursive logit (RL)~\citep{fosgerau2013link}, nested RL~\citep{mai2015nested}, and recursive network GEV~\citep{mai2016method}. However, maximum likelihood estimation of these recursive models is challenging, as it requires a bi-level procedure \citep{oyama2023capturing} where the Bellman equations must be solved many times at each evaluation of the likelihood. This can be time-consuming for large networks. Furthermore, a solution for the parameters may fail to exist~\citep{mai2022undiscounted}. \cite{oyama2023capturing} shows that the estimation of recursive route choice models can be facilitated by additional prism constraints, but at the cost of consistency of the parameter estimates. In contrast, the perturbed utility route choice model \citep{Fosgerau2021a} requires only linear regression for model estimation, while it is still capable of capturing correlations and avoids choice set generation.

A number of stochastic traffic assignment models have been proposed based on different route choice models.
\cite{fisk1980some} proposes a path-based constrained optimization formulation for multinomial logit by extending the \cite{beckmann1956studies} formulation with a path entropy term, whose optimal value satisfies the SUE condition.
\cite{sheffi1982algorithm} propose an unconstrained dual formulation that generalizes \cite{fisk1980some}'s formulation with a satisfaction function to allow other distributions. However, neither the function construction nor model estimation is trivial. To account for path correlations, \cite{bekhor1999formulations,bekhor2001stochastic} develop path entropy-based formulations for cross-nested logit (CNL), paired combinatorial logit (PCL) and generalized nested logit (GNL) models. For a given path set, these path-based SUE models are computationally efficient, but depend on a generated path set.

In order to avoid path set generation, \cite{akamatsu1997decomposition} propose a link-based assignment formulation equivalent to \cite{fisk1980some} with infinite path set, which decomposes the path entropy term with respect to link flow. The traveler route choice behavior in \cite{akamatsu1997decomposition} is indeed characterized by the RL model \citep{fosgerau2013link}.
The \cite{akamatsu1997decomposition} formulation will generally predict cyclic flows. This is caused by the Markov property in combination with the fact that it assigns positive probability to all out-going links \citep{akamatsu1996cyclic}. For the same reason, positive flow is assigned on all links in the network, which might be behaviorally questionable and computationally challenging. \cite{oyama2019prism} propose to constrain the Markovian assignment within a \textit{prism-shaped} state-extended subnetwork to limit cyclic flows by generalizing the branch-and-bound method~\citep{prato2006applying}.
\cite{oyama2022markovian} develop a general link-based assignment formulation for network generalized extreme value models (NGEV)~\citep{daly2006general}, which can be rewritten in terms of perturbed utility as
\begin{align}\label{eq:link_NGEV_formulation}
    \min_{x \in \Omega} \underbrace{\sum_a \int_{0}^{x_a} c(m) \,dm}_{C(x)} + \underbrace{\sum_w \theta^w \sum_{i} \sum_{j:(i,j)} x_{ij}^w \ln \left( \frac{x_{ij}^w}{\alpha_{ji}^w \sum_{j:(i,j)} x_{ij}^w}\right)}_{F(x)},
\end{align}
where $\Omega$ is the feasible region, $c(x)$ is the link cost as function of flow $x$, $\theta^w$ is the GEV scale parameter for OD $w$, and $\alpha_{ji}^w$ is the degree of membership of $j$ of nest $i$. Defining $C(x)$ and $F(x)$ as indicated and assuming suitable differentiability, the first-order conditions of this problem are equivalent to the first-order conditions of the individual perturbed utility maximization problems. This shows that the recursive NGEV is a perturbed utility model.

Although the recursive NGEV assignment model~\eqref{eq:link_NGEV_formulation}~\citep{oyama2022markovian} allows a wide range of GEV models, the underlying modeling assumptions lead to positive predicted OD flows on each link in the network and complicated model estimation. In contrast, we propose a stochastic traffic assignment model for a generic convex perturbation with a condition that induces the optimal flow to be zero on most links  \citep{Fosgerau2021a}, i.e. those that are not in the set of optimal paths.

\subsection{Assignment algorithms}
There exists a variety of solution algorithms for the stochastic traffic assignment problem, which find the equilibrium network flows (i.e., the SUE) corresponding to the optimal traveler route choice decisions.
These algorithms typically follow a two-step iterative procedure. The procedure first performs \textit{stochastic network loading} to determine the auxiliary optimal flows for the given costs (e.g., flow-independent logit loading \citet{dial1971probabilistic}); then updates the network flow \textit{assignment} based on the auxiliary solution  (e.g., the method of successive averages \citep[MSA,][]{powell1982convergence}) and updates the costs. Many GEV-based SUE models exploit a closed-form probability expression for performing efficient network loading~\citep[e.g.,][]{dial1971probabilistic,bekhor2009path,watling2018stochastic,oyama2022markovian}.
A key contribution of this paper is the derivation of a closed-form stochastic network loading function for a generic convex perturbation.

The \textit{stochastic network loading} step essentially provides the steepest descent direction for the stochastic traffic assignment problem. In the assignment step, flows are then updated with a suitable step size.
The well-known MSA uses a predetermined step-size sequence for iteratively updating flows~\citep{sheffi1982algorithm}. Several studies improve the convergence of MSA by using an adaptive step size scheme. For example, \cite{liu2009method} propose a self-regulated step size scheme that reduces the step size if the iterate tends to diverge. \cite{du2021faster} adapt the Barzilai-Borwein step size as a \textit{quasi-Newton} method for path-based SUE, which is determined by solving the secant equation with the step size being an approximation to the inverse Hessian matrix. On the other hand, gradient descent with momentum, accumulating the past gradients to guide the current iterate ~\citep{polyak1964some}, has been shown to speed up convergence \citep{lecun2015deep}. One variant of the momentum method, the accelerated gradient descent (AGD) method~\citep{nesterov1983method,beck2009fast,sutskever2013importance}, has demonstrated its efficiency in solving the recursive NGEV SUE problem~\citep{oyama2022markovian}. However, the original AGD method often exhibits an oscillatory convergence pattern, which detracts from its performance. \cite{chambolle2015convergence} proposed a modification of the AGD algorithm (termed AGD* in this paper) with superior practical performance and an improved theoretical convergence rate~\citep{attouch2016rate}.

Both AGD and AGD* typically require a backtracking line search procedure to find the step sizes. In this paper, we propose a quasi-Newton extension of the AGD* method that automatically scales the gradient with the Hessian diagonal, thus not only improving convergence but also avoiding any line search.

Most of the existing stochastic assignment algorithms consider the primal SUE formulation~\citep[e.g.,][]{fisk1980some,akamatsu1997decomposition,bekhor2001stochastic,oyama2019prism}, in which the decision variables are flows. The primal formulation requires the flows to satisfy conservation constraints, which can be handled with path-based gradient projection~\citep{jayakrishnan1994faster,bekhor2005investigating}. However, the projection in link flow space is not trivial. In contrast, the dual formulation reduces to an unconstrained optimization problem~\citep{daganzo1982unconstrained}, which allows efficient algorithms to be used for the link-based stochastic traffic assignment problems. Adding to existing studies with link-based dual formulations for GEV-based SUE~\citep{maher1998algorithms,xie2012stochastic,oyama2022markovian}, we propose a dual formulation for link-based perturbed utility stochastic traffic assignment problem with \textit{node potentials} being the decision variables, which typically have a smaller dimension than the link variables in a real network.

\section{The perturbed utility route choice model}\label{sec:PURC}
We begin by revisiting the perturbed utility route choice (PURC) model~\citep{Fosgerau2021a}. A network is given by $(\mathcal{V},\mathcal{E})$, where $\mathcal{V}$ is the set of nodes with typical element $v$, and  $\mathcal{E}$ is the set of directed edges (or links) with typical element $(i,j)$ for a link from node $i$ to node $j$.  We assume there exists at least one path between any two network nodes, i.e.
\begin{assumption}\label{A:3}
 The network $(\mathcal{V},\mathcal{E})$ is connected.
\end{assumption}

The node-link incident matrix $A\in \mathbb{R}^{|\mathcal{V}| \times |\mathcal{E}|}$ has entries
\begin{align*}
    a_{v, ij}=\begin{cases}
        -1,& v=i,\\
    1, &v=j,\\
    0, &\text{otherwise}.
    \end{cases}
    \end{align*}
A unit demand vector $b\in \mathbb{R} ^{|\mathcal{V}|}$ is given as
\begin{align*}
    b_v =
    \begin{cases}
        -1, & \text{if $v$ is the traveler's origin},\\
         1, & \text{if $v$ is the traveler's destination},\\
         0, &\text{otherwise}.
    \end{cases}
\end{align*}
A network link flow vector $x\in \mathbb{R}_+^{|\mathcal{E}|}$  satisfies flow conservation if $Ax=b$. 

A traveler in the PURC model is associated with a demand vector $b$, a vector of positive link costs\footnote{For simplicity, we deviate slightly from \citet{Fosgerau2021a} by not making explicit the dependence of link costs on link lengths. This is just a question of notation. } $c\in \mathbb{R}_{++}^{|\mathcal{E}|}$, and link-specific convex perturbation functions $F_e$, where

\begin{assumption}\label{A:1}
The perturbation function $F_e:\mathbb R_+ \rightarrow \mathbb R_+, \forall e \in \mathcal{E}$, is continuously differentiable, strictly convex, and strictly increasing, with $F_e(0)=F'_e(0)=0$ and range equal to $\mathbb R_+$. Define $(F'_e)^{-1} (y)=0$ for $y<0$ such that the inverse function of the derivative of the perturbation function has domain equal to $\mathbb R$.
\end{assumption}
For a flow vector $x\in \mathbb{R}_+^{|\mathcal{E}|}$, we define
\begin{align}\label{eq:sum_link_perturbations}
    F(x) = \sum_{e\in \mathcal{E}}  F_e(x_e)
\end{align} for the sum across links of perturbations $F_e(x_e)$.\footnote{\citet{Fosgerau2021a} defines link-specific perturbation functions by multiplying a basic perturbation function by link lengths. This ensures that the overall perturbation function $F$ is invariant with respect to link splitting.}  

The PURC model assumes that the traveler chooses link flow vector $x$ to minimize a perturbed cost function $c^\top x + F(x)$, under the flow conservation constraint.\footnote{In the route choice context, it is equivalent but more natural to talk about cost minimization rather than utility maximization. } Thus, the traveler's demand solves the following convex program.
\begin{subequations}\label{eq:PURC}
    \begin{align}
    \min_{x \in \mathbb{R}_+^{|\mathcal{E}|}} & c^\top x + F(x)\label{eq:PURC_objective}\\
    \text{s.t.}\quad & Ax = b.\label{eq:PURC_onservation}
\end{align}
\end{subequations}

It is an important feature of the program \eqref{eq:PURC} that the objective \eqref{eq:PURC_objective} is convex and separable by links. The coupling across links arises only through the linear conservation constraint \eqref{eq:PURC_onservation}. \cite{Fosgerau2021a} exploit this property to derive an estimation procedure that requires only linear regression.

\subsection{PURC solution properties}\label{sec:purc_solution_property}
To develop our fast assignment algorithm, we must establish some properties of the solution to the traveler's route choice problem. The program \eqref{eq:PURC} is convex with a strictly convex objective; hence, the solution exists uniquely. We indicate the optimal values of decision variables by `\textsuperscript{*}'. The complementarity condition for problem \eqref{eq:PURC} is
    \begin{align}\label{eq:comp_PURC}
         0 \leq x^*_{ij} \perp \left[ c_{ij} + F_{ij}'(x^*_{ij})+ \eta^*_j - \eta^*_i \right] \geq 0,\; \forall (i,j) \in \mathcal{E},
    \end{align}
where $\eta = \{\eta_v\}_{v\in \mathcal{V}}$ denotes the dual variable for the conservation constraint \eqref{eq:PURC_onservation}. The dual variable $\eta$ enters only as a difference $\eta_j^*-\eta_i^*$; hence it is determined only up to a constant. We are therefore free to impose the constraint that $\eta_d = 0$ at the traveler's destination node $d$.

We can derive an expression for the optimal link flows as a function of the dual variables $\eta_j$. For $x_{ij}^{*} > 0$, we have by complementarity and using the invertibility of $F_{ij}'$ that
    \begin{align}\label{eq:xij}
        x_{ij}^{*} = (F_{ij}')^{-1}\left(\eta_i^* - \eta_j^* - c_{ij}  \right), \; \forall (i,j) \in \mathcal{E}.
\end{align}
For $x^*_{ij}=0$, under Assumption~\ref{A:1}, we have similarly that
\begin{equation*}
    0 \leq F_{ij}'(x^*_{ij})+ c_{ij} + \eta^*_j - \eta^*_i = c_{ij} + \eta^*_j - \eta^*_i,
\end{equation*}
and hence Eq.~\eqref{eq:xij} holds also in this case. We will use this in the next section to find a closed-form network loading expression for the dual Lagrangian of the traffic assignment problem that we will formulate.

The dual variables $\eta^*_j$ have a very immediate and useful interpretation as they are equal to the shortest path marginal cost from node $j$ to the destination. To see this, note that $\eta^*_i \leq F_{ij}'(x^*_{ij})+ c_{ij} + \eta^*_j $ for all $(i,j)\in \mathcal{E}$. Then, by the complementarity condition \eqref{eq:comp_PURC}\begin{align}\label{eq:etamin}
    \eta^*_i = \min_{j :(i,j)\in \mathcal{E}} \left\{ \eta^*_j + F_{ij}'(x^*_{ij}) + c_{ij} \right\}.
\end{align}

\section{The perturbed utility-based traffic assignment problem}\label{sec:PURC_TAP}
We now set up the traffic assignment problem for the perturbed utility route choice model. We consider a general assignment setting with multiple traveler types and allow for arbitrary heterogeneity. We denote the set of traveler types by $\mathcal{W}$ with typical element $w$ and the volume of travelers of each type $w$ by $q^w$.

The network and the link travel times are common across types. Congestion in the network causes interaction between travelers. Otherwise, each type is assumed to behave according to its own perturbed utility route choice model \eqref{eq:PURC} as described in the previous section. We index the link cost functions, the perturbation functions, and the demand vectors by the type, $c_{ij}^w,F_{ij}^w$, and $b^w$.
The link cost function $c_{ij}^w$ is a type-specific function of link travel time $t_{ij}$. This allows e.g. type-specific preferences against tolls, while the link travel time is common to all types. We make the following assumptions.

\begin{assumption}\label{A:2}
The link cost functions $c_{ij}^w:\mathbb R_+ \rightarrow \mathbb R_+, (i,j)\in \mathcal{E}, w\in\mathcal{W} $ are positive, continuously differentiable, and increasing with respect to link flow, and convex, $c_{ij}^w(0)>0$, ${c_{ij}^{w}}'>0$.
\end{assumption}
\begin{assumption}\label{A:4}
    The link travel time functions $t_{ij}:\mathbb{R}_+\rightarrow \mathbb{R}_+,(i,j)\in \mathcal{E}$ are positive, differentiable, increasing, and strictly convex, $t_{ij}>0, t_{ij}'>0$.
\end{assumption}

The travel time on link $(i,j)$ depends on the link flow $x_{ij} = \sum_{w \in\mathcal{W}} q^w x^w_{ij}$ and the link cost for type $w$ is then $c_{ij}^{w}(t_{ij}(x_{ij}))$. Assumptions \ref{A:2} and \ref{A:4} combine to ensure that the link costs are convex functions of link flow.

\subsection{Primal formulation}

We will formulate a convex minimization problem, whose solution is the Wardrop/Nash equilibrium where all travelers make individually optimal choices according to \eqref{eq:PURC}, taking link costs as given. Link costs, in turn, depend through the common travel time on the link flow, which is the aggregated demand from the individual travelers. This is our primal perturbed utility-based stochastic traffic assignment problem (TAP).

\vspace{1cm}

\begin{subequations}
[TAP]
\label{eq:assignment}
  \begin{align}
    \min_{\bm{x}} Z = \quad & \sum_{(i,j) \in \mathcal{E}} \sum_{w \in \mathcal{W}} \left[ \int_{0}^{\sum_{w' \in \mathcal{W}}q^{w'}x_{ij}^{w'}} c_{ij}^w(t_{ij}(m)) dm + q^w F_{ij}^w(x_{ij}^w)  \right] \quad  & \label{eq:primal_TAP_obj}\\
        s.t. \quad & x_{ij}^w \geq 0, \; \forall (i,j) \in \mathcal{E}, w \in \mathcal{W} &  \label{eq:x_primal_feasibility_1}\\
    & q^w (A_v \bm{x}^w - b_v^w) = 0, \; \forall v \in \mathcal{V}, w \in \mathcal{W}.& (\eta_v^w) \label{eq:x_primal_feasibility_2}
  \end{align}
\end{subequations}

In this expression, $\bm{x}^w = (x_e^w)_{e\in\mathcal{E}}$ is the link flow vector for type $w$, and $\bm{x} = \{x_e^w\}_{e\in\mathcal{E},w\in\mathcal{W}}$. $A_v$ is the row vector corresponding to node $v$ of the incidence matrix $A$.
Eq.~\eqref{eq:x_primal_feasibility_2} is the conservation constraints with corresponding Lagrange multipliers $\eta_v^w$ for each node and type. We denote $\bm{\eta}= (\eta_v^w)_{v\in \mathcal{V}, w\in \mathcal{W}}$.

Our proposed primal TAP formulation extends  \cite{beckmann1956studies}'s deterministic UE formulation with the addition of a perturbation term.
The PURC is similar to the recursive NGEV formulation of \citet{oyama2022markovian} with the important difference that PURC allows corner solutions, i.e. links with zero flow. In particular, because of the requirements on the PURC perturbation function, the optimal flow for a given OD concentrates on a relatively small number of paths, while it is zero in the rest of the network. 

We first show that the equilibrium link flow pattern in the primal TAP equals the PURC demand for each $w$.
\begin{proposition}
    The stochastic user equilibrium condition in TAP~\eqref{eq:assignment} is equivalent to the optimality condition in PURC~\eqref{eq:PURC}, such that traveler route choice behavior is in accordance with the PURC model.
\end{proposition}
\begin{proof}
Denote optimal link flows by $x^*_{ij}=\sum_{w \in \mathcal{W}}q^w x_{ij}^{w*}$ and equilibrium type-specific link costs by $c_{ij}^{w*}=c_{ij}^w(t(x^*_{ij}))$
The complementarity condition for $x_{ij}^w$ in the primal TAP is

\begin{align}
\label{eq:FOC_primal_TAP}
    0 \leq x_{ij}^{w*} & \perp q^w(c_{ij}^{w*} + F_{ij}'(x_{ij}^{w*}) + \eta_j^{w*} - \eta_i^{w*}) \geq 0.
\end{align}
Since we consider only types with positive demand $q^w > 0$, this reduces to
\begin{align}
    0 \leq x_{ij}^{w*} & \perp (c_{ij}^{w*} + F_{ij}'(x_{ij}^{w*}) + \eta_j^{w*} - \eta_i^{w*}) \geq 0,
    \label{eq:complementarity}
\end{align}
which is the same as Eq.~\eqref{eq:comp_PURC} required for each traveler.
\end{proof}

We are now ready to prove the existence and uniqueness of the solution to the primal TAP.
\begin{proposition}[Existence and uniqueness of TAP]\label{lemma:TAP_existence_uniqueness}

The primal traffic assignment problem (TAP) admits a unique solution.
\end{proposition}
\begin{proof}
By Assumptions \ref{A:1}, \ref{A:2}, and \ref{A:4} on $F^w_{ij}$, $c^w_{ij}$, and $t_{ij}, ij\in\mathcal{E}, w\in\mathcal{W}$, the objective $Z$ is strictly convex. In addition, the domain is convex. Hence, any solution is unique. To establish existence, we will show that the relevant domain is bounded. Observe that the perceived link cost $c_{ij}^w(x_{ij}) + F_{ij}'(x_{ij}^w)$ is positive for all links and types. Links with positive flow can never form a cycle at optimum, since then a cheaper feasible flow vector can be identified by removing the cyclic flow. It is therefore no loss of generality to restrict the domain of $Z$ to the set of acyclic link flows for each $w$. This set is bounded, hence TAP~\eqref{eq:assignment} admits a unique solution.
\end{proof}

\subsection{Dual formulation}
The primal TAP involves a large number of flow conservation constraints, one for each combination of type and network node. In this section, we will utilize that we have found an expression for the individual link flows as a function of the dual parameters. We will use this to find a closed-form expression for the dual problem corresponding to the primal TAP, and we will show that the dual problem is unconstrained. This will be useful for finding a fast solution algorithm.

The Lagrangian function for the primal TAP  \eqref{eq:assignment} is

\begin{subequations}

\label{eq:assignment_lagrangian}
    \begin{align}
    L(\bm{x, \eta}) = &  \sum_{(i,j) \in \mathcal{E}} \sum_{w \in \mathcal{W}} \left[\int_{0}^{\sum_{w' \in \mathcal{W}}q^{w'} x_{ij}^{w'}} c_{ij}^w(t_{ij}(m)) dm + q^w F_{ij}^w(x_{ij}^w)  \right] \nonumber \\
    & - \sum_{(i,j) \in \mathcal{E}}\sum_{w \in \mathcal{W}} q^w\left(\eta_i^w - \eta_j^w \right)x_{ij}^w - \sum_{v \in \mathcal{V}}\sum_{w \in \mathcal{W}} q^w \eta_v^w  b_v^w \label{eq:dual_TAP_obj}\\
    \text{s.t.} \quad& x_{ij}^w \geq 0,\; \forall (i,j) \in \mathcal{E}, w \in \mathcal{W},
\end{align}
\end{subequations}
where $\eta^w_{i},w\in \mathcal{W}, i\in \mathcal{V}$ are the Lagrangian multipliers for the flow conservation constraints. As for the individual traveler, we are free to impose the normalization that $\eta^w_{d^w}=0$, such that these Lagrangian multipliers can be interpreted as the minimum perceived cost from node $i$ to the destination node $d^w$ of each type $w$.

The Lagrangian function \eqref{eq:assignment_lagrangian} has simple constraints, but adds extra decision variables, the dual variables $\eta^w_{ij}$. Using the closed-form expression \eqref{eq:xij} for the flow variables $x^{w*}_{ij}$ as a function of the dual variables allows us to reduce the number of decision variables considerably.

\begin{proposition}[Perturbed utility-based network loading]
The optimal link flow $x_{ij}^{w*}$ for given $\eta$ is
\begin{align}
 x_{ij}^{w*} = x_{ij}^{w*}(\eta_i^{w*}, \eta_j^{w*},c^{w*}_{ij}) = (F_{ij}^{w'})^{-1}\left(\eta_i^{w*} - \eta_j^{w*} - c^{w*}_{ij} \right), \label{eq:closed_form_flow}
\end{align}
where
\begin{align}
c^{w*}_{ij} = c^w_{ij}\left(t_{ij}\left(\sum_{w \in \mathcal{W}}q^w x^{w*}_{ij}\right)\right). \label{eq:fixed_point_link_cost}
\end{align}
\end{proposition}
\begin{proof}
As shown in Section~\ref{sec:purc_solution_property}, Eq.~\eqref{eq:xij} holds for $x_{ij}^{w*} \geq 0$. This is ensured by projecting the RHS of
Eq.~\eqref{eq:closed_form_flow} onto the positive orthant. Consequently, the complementarity conditions~\eqref{eq:complementarity} are satisfied for given $\eta^*$ and the corresponding $x_{ij}^{w*}$. Hence, the optimal link flow is obtained by Eq.~\eqref{eq:closed_form_flow}.
\end{proof}

Having determined $x^{w*}_{ij}$ as a function of $\eta_i^{w*}, \eta_j^{w*}$ and $c_{ij}^{w*}$ also in the primal TAP, we can substitute that into the  Lagrangian~\eqref{eq:assignment_lagrangian} to obtain the corresponding Lagrangian dual and the dual traffic assignment problem (DTAP).

[DTAP]
\begin{align}
\label{eq:assignment_lagrangian_dual}
    \max_{\bm{\eta}} G = & \sum_{(i,j) \in \mathcal{E}} \sum_{w\in\mathcal{W}}\left[ \int_{0}^{\sum_{w' \in \mathcal{W}}q^{w'} x_{ij}^{w'*}} c^{w}_{ij}(t_{ij}(m)) dm + q^w  F_{ij}^w(x_{ij}^{w*})  \right] \nonumber \\
    & - \sum_{(i,j) \in \mathcal{E}}\sum_{w \in \mathcal{W}} q^w\left(\eta_i^w - \eta_j^w \right)x_{ij}^{w*} - \sum_{v \in \mathcal{V}}\sum_{w \in \mathcal{W}} q^w \eta_v^w  b_v^w
\end{align}
We note that DTAP is unconstrained with the node potentials $\bm{\eta}$ being the only decision variables. This allows us to adapt existing fast algorithms for solving the DTAP.  The closed-form network loading expression~\eqref{eq:closed_form_flow}, allows us to directly obtain the corresponding individual flows.
The following lemma shows that the strong duality condition holds, such that solving the dual problem is equivalent to solving the primal TAP~\eqref{eq:assignment}.

\begin{lemma}[Strong duality]\label{lemma:strong_duality}
The duality gap between the primal problem TAP \eqref{eq:assignment} and the corresponding dual problem at their optimal solutions is zero.
\end{lemma}
\begin{proof}
The primal TAP~\eqref{eq:assignment} is convex, the constraints~\eqref{eq:x_primal_feasibility_1}-\eqref{eq:x_primal_feasibility_2} are linear, and optimal solution $x^*$ exists. Then we need only to verify the weak Slater's condition \citep[][Section 5.2.3]{boyd2004convex}, namely that
\begin{align}
\exists x^w \in \mathbb{R}^{|\mathcal{E}|}: \quad - x^w \leq 0, \; Ax^w - b^w = 0, \forall w \in \mathcal{W}. \nonumber
\end{align}
However, this condition is satisfied by the flow determined by all-or-nothing (shortest path) assignment. We conclude that the duality gap is zero and thus that strong duality holds.
\end{proof}

Finally, we need to verify that the DTAP admits a solution, such that we can always use the unconstrained DTAP to solve the perturbed utility-based stochastic traffic assignment problem.
\begin{lemma}[Existence of DTAP] \label{lemma:DTAP_existence}
The DTAP \eqref{eq:assignment_lagrangian_dual} admits at least one solution $\eta^*$.
\end{lemma}
\begin{proof}
By Proposition~\ref{lemma:TAP_existence_uniqueness}, the optimal flow $x_{ij}^{w*}$ exists uniquely and link costs $c_{kj}^w(t_{ij}(x_{ij}^*)) + (F_{ij}^{w})'(x_{kj}^{w*})$ are positive and bounded. Hence, by the RHS of the complementarity~\eqref{eq:complementarity}, $\eta^{w*}_{k}$ is bounded up to a constant $\eta^w_{d^w}$. In addition, for the given $x_{ij}^{w*}$, DTAP~\eqref{eq:assignment_lagrangian_dual} reduces to a linear program. We conclude that DTAP admits a solution $\eta^*$ for any given constant $\eta^w_{d^w}$.
\end{proof}

\section{Solution method}\label{sec:solution_algo}
In this section, we propose a fast algorithm for solving the dual assignment problem~\eqref{eq:assignment_lagrangian_dual}. In contrast to the existing primal algorithms that solve for optimal link flows as a constrained optimization problem, the proposed dual approach solves for the Lagrangian multipliers (interpreted as the minimum perceived costs) in an unconstrained optimization problem. In addition to avoiding constraints, the dual problem has the advantage that the number of decision variables is smaller, since the number of nodes is typically less than half the number of links in a road network $({|\mathcal{V}|}/{|\mathcal{E}|} \leq 0.5)$.

Accelerated gradient descent (AGD)~\citep{beck2009fast} is a first-order method with an additional Nesterov momentum~\citep{nesterov1983method} step for extrapolation, which has demonstrated its usefulness for solving origin-based Markovian traffic equilibrium problems~\citep{oyama2022markovian}. \citet{oyama2022markovian} require a backtracking procedure to successively reduce the step size in order to obtain algorithm convergence. This means each iteration requires several runs of network loading to select the step size, which is computationally costly for large networks. Moreover, the convergence trajectory in the original AGD often oscillates at later iterations, which reduces the speed of the algorithm. We propose a quasi-Newton accelerated gradient descent algorithm (qN-AGD*), which uses the AGD* scheme~\citep{chambolle2015convergence} to reduce oscillation, and uses the Hessian diagonal to automatically scale the gradient in a quasi-Newton manner with fixed step size, without the need for a backtracking procedure.

Recall that the optimal link flows \eqref{eq:closed_form_flow} depend on the type-specific link costs $c^{w*}_{ij}$. However, these in turn depend on the link flows through the link travel time functions $t_{ij}$. To tackle this problem, existing primal algorithms adapt the Gauss-Seidel approach, which decomposes the assignment problem for each origin-based network flow and iteratively solves these subproblems~\citep[e.g.,][]{dial2006path,nie2010class}. In contrast, we consider all the node potential variables as one block of variables, while the link travel times are considered as another separate block of variables. Thus the assignment problem is decomposed into only two subproblems. This exploits that all node potentials $\eta$ can be updated in parallel, while the link travel times $t$ can be updated subsequently.

We cast the link travel time problem as an auxiliary fixed-point problem for given $\eta$.
\begin{align*}
t_{ij} &= t_{ij}\left(\sum_{w \in \mathcal{W}}q^w \cdot (F_{ij}^{w'})^{-1}\left(\eta_i^{w*} - \eta_j^{w*} - c^{w*}_{ij} \right)\right)
\end{align*}
with corresponding residual function
\begin{align*}
U_{ij}(x_{ij}^{*}, t^*_{ij}) &= t_{ij}\left(\sum_{w \in \mathcal{W}}q^w x_{ij}^{w*} \right) - t^*_{ij} = 0
\end{align*}

We find that the fixed-point $t^*$ exists.
\begin{proposition}
Under Assumption~\ref{A:4}, the fixed point $t^*_{ij},(i,j)\in \mathcal{E}$ exists.
\end{proposition}
\begin{proof}
By Proposition~\ref{lemma:TAP_existence_uniqueness}, type-specific link flows $x_{ij}^{w*}$ exist uniquely. Hence, link flows $x_{ij}^{*} = \sum_{w \in \mathcal{W}}x_{ij}^{w*}$ are also unique. The strict monotonicity of the  link travel time functions $t_{ij}(x_{ij})$ then implies that the fixed point $t_{ij}^* = t_{ij}(x_{ij}^*)$ exists.
\end{proof}

We now propose the following dual assignment algorithm, combining the AGD* algorithm with a quasi-Newton gradient scaling, as well as a Newton step for updating link travel times.
\begin{algorithm}[H]
\caption{Dual assignment algorithm -- qN-AGD*}\label{alg:dual_algorithm}
\begin{algorithmic}
\State \textbf{Step 0: Initialization}.\; Input initial points $\eta^{(0)}$ and $c^{*(0)}$, and step sizes $\gamma_1, \gamma_2$, set iteration counter $m=0$, momentum acceleration variable $r_0=1$ and momentum acceleration parameter $\alpha > 1$ .\\

\State \textbf{Step 1: Iteration}.\;
\State  \quad\textit{\textbf{Step 1.1.}}\;- \textit{PURC assignment}:\; \begin{align}\label{eq:PURC_assignment_projection}
x_{ij}^{w*(m+1)} = \min \left\{1, (F_{ij}^{w'})^{-1}\left(\eta_i^{w(m)} - \eta_j^{w(m)} - c^{w*(m)}_{ij} \right)\right\}.
\end{align}
\State  \quad\textit{\textbf{Step 1.2.}}\;- \textit{Update Lagrangian multipliers $\eta$, with Nesterov's momentum acceleration}:\;
\begin{subequations}\label{eq:node_potential_update}
\begin{align}
\tilde{\eta}_j^{w(m+1)} &= \eta_j^{w(m)} + \gamma_1 \tilde{\nabla}_{\eta_j^{w}} G(\eta^{(m)}),\label{eq:aux_node_potential_update}\\
{\eta}_j^{w(m+1)} &= \tilde{\eta}_j^{w(m+1)} + \frac{m}{m + \alpha} \left(\tilde{\eta}_j^{w(m+1)} - \tilde{\eta}_j^{w(m)} \right).
\end{align}
\end{subequations}

\State  \quad\textit{\textbf{Step 1.3.}}\;- \textit{Update auxiliary link travel times $t^{*}$ and $c^{w*}$, by one Newton step}:\; \begin{align}\label{eq:travel_time_newton}
t^{*(m+1)}_{ij} = t^{*(m)}_{ij} - \gamma_2 \frac{U_{ij}\left(x_{ij}^{w*(m+1)}, t^{*(m)}_{ij}\right)}{\nabla_{t^{*}_{ij}} U_{ij}\left(x_{ij}^{w*(m+1)}, t^{*(m)}_{ij}\right)}.
\end{align} \\
\State \quad \quad \quad \quad \quad \quad In addition, update link cost:\; \begin{align}
    c^{w*(m+1)} = c^{w}\left(t^{*(m+1)}_{ij} \right).
\end{align}

\State \textbf{Step 2: Convergence test}.\; If the stopping criteria hold, \textbf{stop}. Otherwise, set $m = m +1$ and go to \textbf{Step 1}.\\

\end{algorithmic}
\end{algorithm}

Both AGD and AGD* employ the Nesterov momentum acceleration method, while mainly differing in the momentum updating scheme for $r_{m+1}$. Compared to the AGD, AGD* Eq.~\eqref{eq:node_potential_update} uses a smaller extrapolation step, which is often slower at early iterations but has the advantage of avoiding potential oscillation when approaching convergence~\citep{chambolle2015convergence}. We will perform numerical experiments to illustrate the superior performance of the AGD* to AGD in the following section.

Next, we will show how to compute the gradient of the Lagrangian dual $G(\bm{\eta})$. By the envelope theorem \citep{milgrom2002envelope}, for the given link costs $c^*$ and $x^{w*}_{ij}$ (by Eq.~\ref{eq:closed_form_flow}), the gradient of $G(\bm{\eta})$ w.r.t. $\eta_k^w$ is

\begin{align}\label{eq:simplified_gradient}
{\nabla}_{\eta_k^{w}} G({\eta}^{(m)}) &= q^w\left[\sum_{i:(i,k)\in\mathcal{E}}  x^{w*(m+1)}_{ik} - \sum_{j:(k,j)\in\mathcal{E}}x^{w*(m+1)}_{kj} - b_k^w\right] \nonumber\\
&= q^w (A_k x^{w*(m+1)} - b_k^w).
\end{align}

We further propose to scale the gradient with an upper bound of the Hessian diagonal to speed up convergence \citep{nie2012note}. This is often considered as a quasi-Newton method in primal deterministic assignment algorithms~\citep[e.g.,][]{bar2002origin,nie2010class}. The corresponding scaled gradient is
\begin{align}\label{eq:scaled_gradient}
\tilde{\nabla}_{\eta_k^{w}} G({\eta}^{(m)}) &= \frac{{\nabla}_{\eta_k^{w}} G({\eta}^{(m)})}{q^w A_k  \left(\nabla_{\eta_k^{w(m)}}((F_{ij}^{w})')^{-1}\right)^{|\mathcal{E}|\times 1}},
\end{align}
where the division is performed element-wise. Note that, in the qN-AGD* algorithm, the scaled gradient~\eqref{eq:scaled_gradient} is used to update auxiliary node potentials $\tilde{\eta}$ in Eq.~\eqref{eq:aux_node_potential_update}, which makes it a quasi-Newton method. This is in contrast to the AGD and AGD* algorithms which are first-order methods that do not utilize any Hessian.

For the auxiliary fixed-point problem, we apply a Newton step~\eqref{eq:travel_time_newton} for updating the travel times. The gradient of $U_{ij}$ is
\begin{align}
\nabla_{c^{*}_{ij}} U_{ij} = \frac{\partial c_{ij}}{\partial x_{ij}} \sum_{w \in \mathcal{W}} q^w \nabla_{c_{ij}^{w*(m)}}((F_{ij}^{w})')^{-1}.
\end{align}
We note that the resulting dual assignment algorithm (qN-AGD*) resembles the coordinate descent method, whose convergence properties are well-established~\citep{tseng2001convergence}. Furthermore, \textit{\textbf{Step 1.1.}} takes an additional projection of $x_{ij}^{w*(m+1)}$ onto $x_{ij}^{w*(m+1)} \leq 1$, corresponding to the feasible network flows of a traveler. Consequently, the gradient of $\eta_k^w$ (Eq.~\ref{eq:simplified_gradient}) is clipped at $q^w(\sum_{(i,j) \in \mathcal{E}}a_{ij} - b_k^w)$. As shown in \cite{zhang2019gradient}, as the local Lipschitz constants are often larger at early iterations, gradient clipping can accelerate algorithm convergence by adaptively scaling the gradient (to the clipped values), such that the global Lipschitz smoothness assumption can be relaxed. To this end, gradient clipping is applied at the early iterations (e.g., first 100), while it allows us to use a relatively large fixed step size without the need for a backtracking line search.

We consider the algorithm to have converged when the following two criteria both hold:
\begin{align}
R_1^{(m+1)} = \sum_{w\in\mathcal{W}} \frac{q^w}{Q}\frac{||Ax^{w*(m+1)} - b^w||_{1}}
{|\mathcal{V}|} \leq \epsilon_1 \; \text{and}\; R_2 = \frac{||t\left(x^{w*(m+1)}\cdot q^{|\mathcal{W}|\times 1} \right) - t^{*(m+1)}||_1}{|\mathcal{E}|} \leq \epsilon_2, \nonumber
\end{align}
where $||\cdot||_1$ denotes the $L_1$-norm, total demand is $Q=\sum_{w\in\mathcal{W}}q^w$, and $t(\cdot)$ is the vector function for link travel times depending on link flows. $R_1$ is the mean absolute error of the first-order condition for $\eta^*$ weighted by demands, and $R_2$ is the mean absolute error of $t^*$ to the auxiliary fixed point.

\subsection{Entropy perturbation function}
We now illustrate the gradient computations for the entropy perturbation function proposed in \cite{Fosgerau2021a}:
\begin{align}\label{eq:entropy_PURC}
    F_{ij}^w(x_{ij}^w) = (1 + x_{ij}^w)\ln(1 + x_{ij}^w) - x_{ij}^w, \forall (ij) \in \mathcal{E}, w \in \mathcal{W}.
\end{align}
The optimal link flows are
\begin{align}
x_{ik}^{w*} = \left[\exp\left(\eta_i^{w*} -c_{ik}^* - \eta_k^{w*}\right) -1\right]^+.
\end{align}
The scaled gradient $\tilde{\nabla}_{\eta_k^{w}} g({\eta}^{(m)})$ simplifies to
\begin{align}
\tilde{\nabla}_{\eta_k^{w}} G({\eta}^{(m)}) &= \frac{A_k x^{w*(m+1)} - b^w_k}{A_k \odot A_k\left(x_{ij}^{w*(m+1)} +1\right)^{|\mathcal{E}|\times 1} },
\end{align}
where $\odot$ is the Hadamard (element-wise) product. The gradient $\nabla_{c^{*}_{ij}} U_{ij} $ is
\begin{align}
\nabla_{c^{*}_{ij}} U_{ij} = -\frac{\partial c_{ij}}{\partial x_{ij}} \sum_{w \in \mathcal{W}} q^w \left(x_{ij}^{w*(m+1)} +1\right) \delta_{ij}^w -1,
\end{align}
where $\delta_{ij}^w =1$ if $x_{ij}^{w*(m+1)} > 0$, and $\delta_{ij}^w =0$, otherwise.

\section{Numerical experiments}
\label{sec:results}
In this section, we demonstrate the performance of our proposed dual assignment algorithm through a series of numerical experiments.

We use the entropy perturbation function \eqref{eq:entropy_PURC}, we initialize $\eta^{(0)}, c^{*(0)}$ with the costs from on all-or-nothing assignment not including the perturbation term, set the momentum acceleration parameter to $\alpha=10$, and
the stopping criteria parameters to $\epsilon_1 = \epsilon_2 = 10^{-5}$. We specify the link cost function simply as $c_{ij}(x_{ij}) = 0.5t_{ij}(x_{ij})$, where $t_{ij}$ is the Bureau of Public Roads (BPR) volume-delay function. For the proposed qN-AGD* algorithm, we fix the step sizes as $\gamma_1 = 0.5, \gamma_2 = 1$, which provide satisfactory performance in all numerical experiments. For reference, we also tested the performances of the AGD method~\citep{beck2009fast, oyama2022markovian}, the AGD* method~\citep{chambolle2015convergence}, and a primal algorithm based on three-operator splitting~\citep[][TOS]{davis2017three}, which allows efficient gradient projection for the primal formulation. For each reference algorithm, we report the best convergence result among step size choices of $\{10^{-4}, 10^{-5}, 10^{-6}\}$. All algorithms are implemented in PyTorch to enable GPU computation and run on a HPC cluster with an A100 GPU. The real networks are obtained from \cite{bargeraNetwork}, with the default volume-delay functions and demand matrix as provided.

\subsection{Convergence performance at different demand levels}
We begin by comparing the convergence performance of the proposed qN-AGD* algorithm to the qN-AGD, AGD* and AGD algorithms, under three demand levels (1q, 1.5q, 2q). For this comparison, we use the Sioux Falls network, which has 76 links, 24 nodes, 528 OD pairs and 360,600 trips (1q).
The key difference between the original AGD and our algorithm AGD* as defined in Algorithm~\ref{alg:dual_algorithm} is that the AGD uses a different momentum, which is $r_{m+1} = 0.5(1 + \sqrt{1 + 4r_{m}^2})$~\citep{beck2009fast}.
\begin{figure}[htb]
    \centering
    \includegraphics[width=0.8\textwidth]{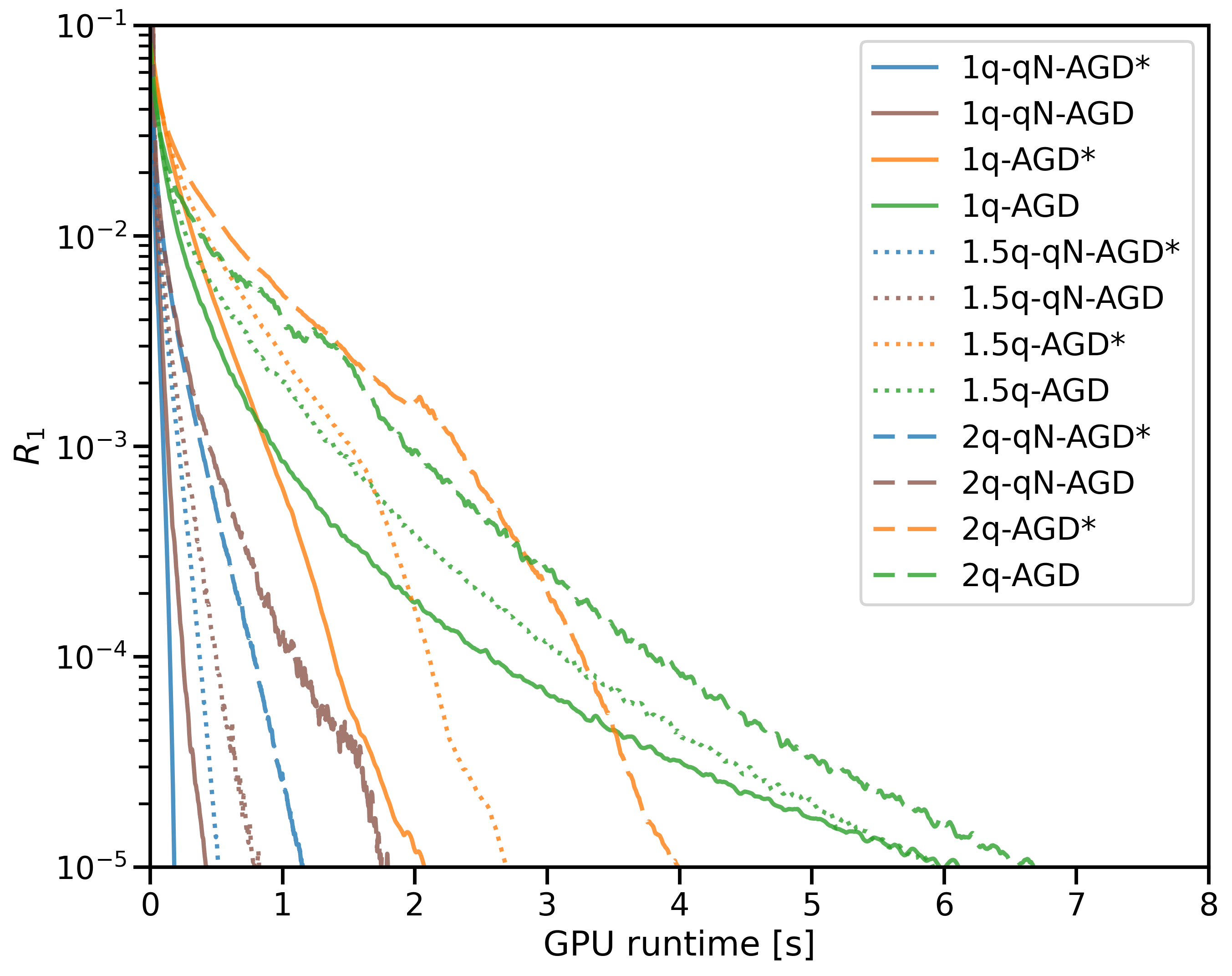}
    \caption{Convergence performance of the proposed quasi-Newton AGD* (qN-AGD*), qN-AGD, AGD*, and the original AGD, with three demand levels (1q, 1.5q, 2q). The horizontal axis shows the GPU runtimes in sec. The vertical axis shows the mean absolute error of the first-order condition.}
    \label{fig:alg_convergence_foc_gap}
\end{figure}

Figure \ref{fig:alg_convergence_foc_gap} shows the convergence of the first-order gap $R_1$ for the four algorithms, under three demand levels.
The proposed qN-AGD* algorithm is the fastest of the four algorithms at any demand level: the qN-AGD* under 2q demand is even faster than AGD* at 1q demand. In addition, the qN-AGD* runtime at different demand levels only has small variations, which suggests the proposed qN-AGD* has the potential to work well also for congested networks. The qN-AGD is generally second best, which indicates that the quasi-Newton scaling is important. 
Among AGD* and AGD, AGD* is faster for higher precision solutions (e.g., with $R_1 = 10^{-5}$), while AGD is faster for lower precision (e.g., with $R_1 = 10^{-2}$). This is because AGD* avoids oscillation with a more conservative updating scheme ($\alpha > 3$), which might be slower at the earlier iterations than AGD (approximated by $\alpha = 3$)~\citep{liang2022improving}.

\begin{figure}[ht]
    \centering
    \subfloat[$1q$]{\includegraphics[width=0.65\linewidth]{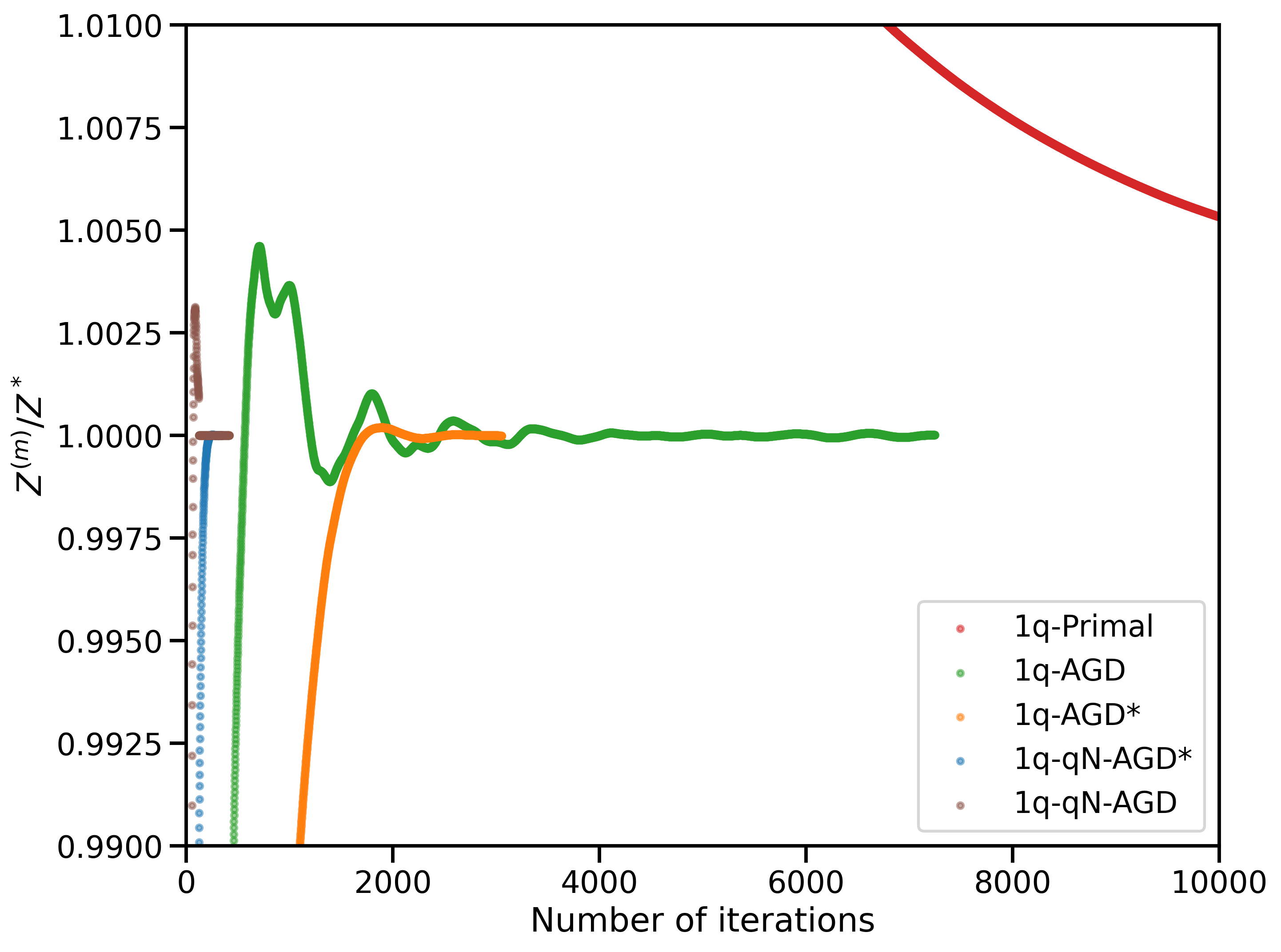}}\\
    \subfloat[$1.5q$]{\includegraphics[width=0.5\linewidth]{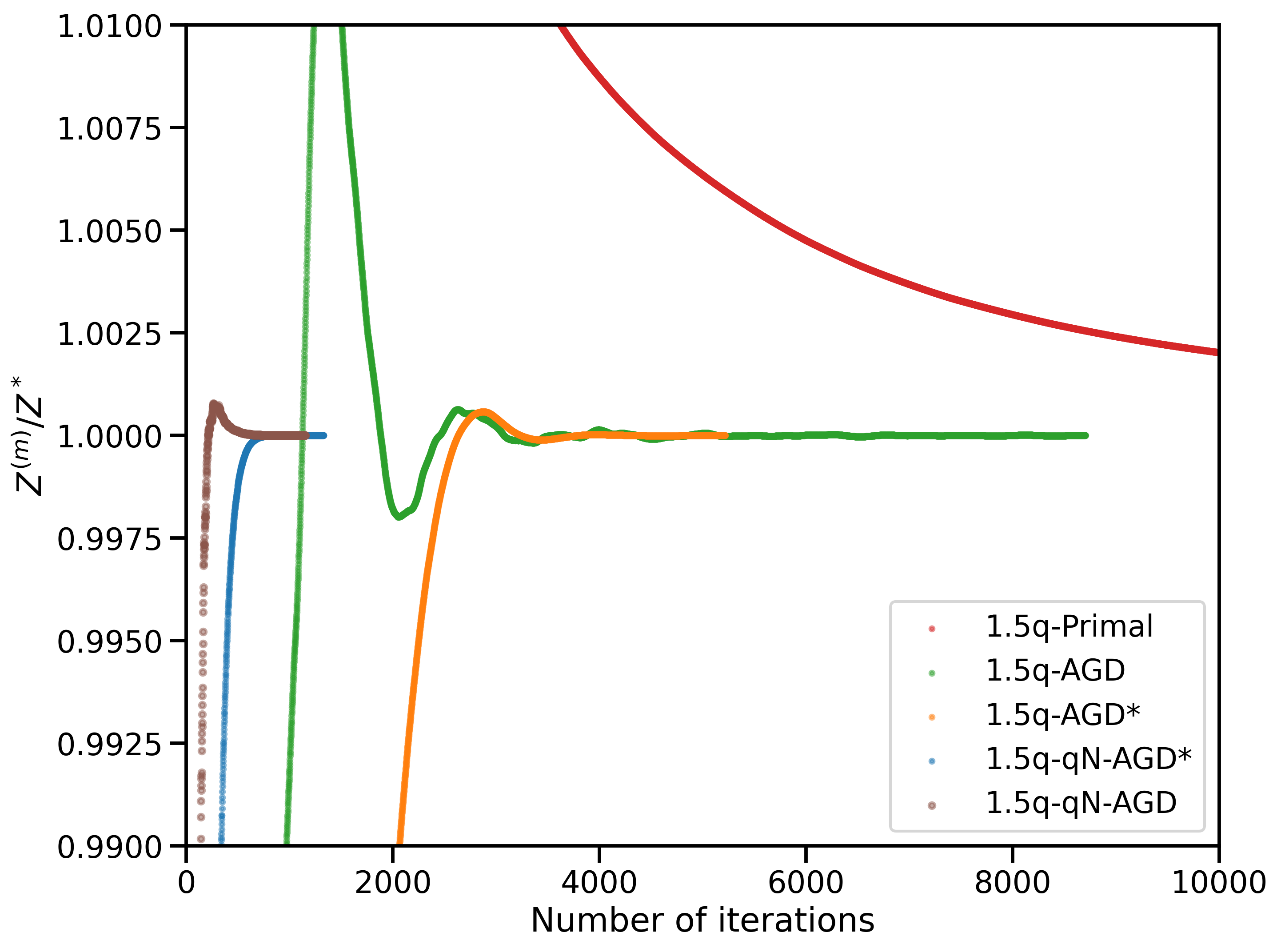}}
    \subfloat[$2q$]{\includegraphics[width=0.5\linewidth]{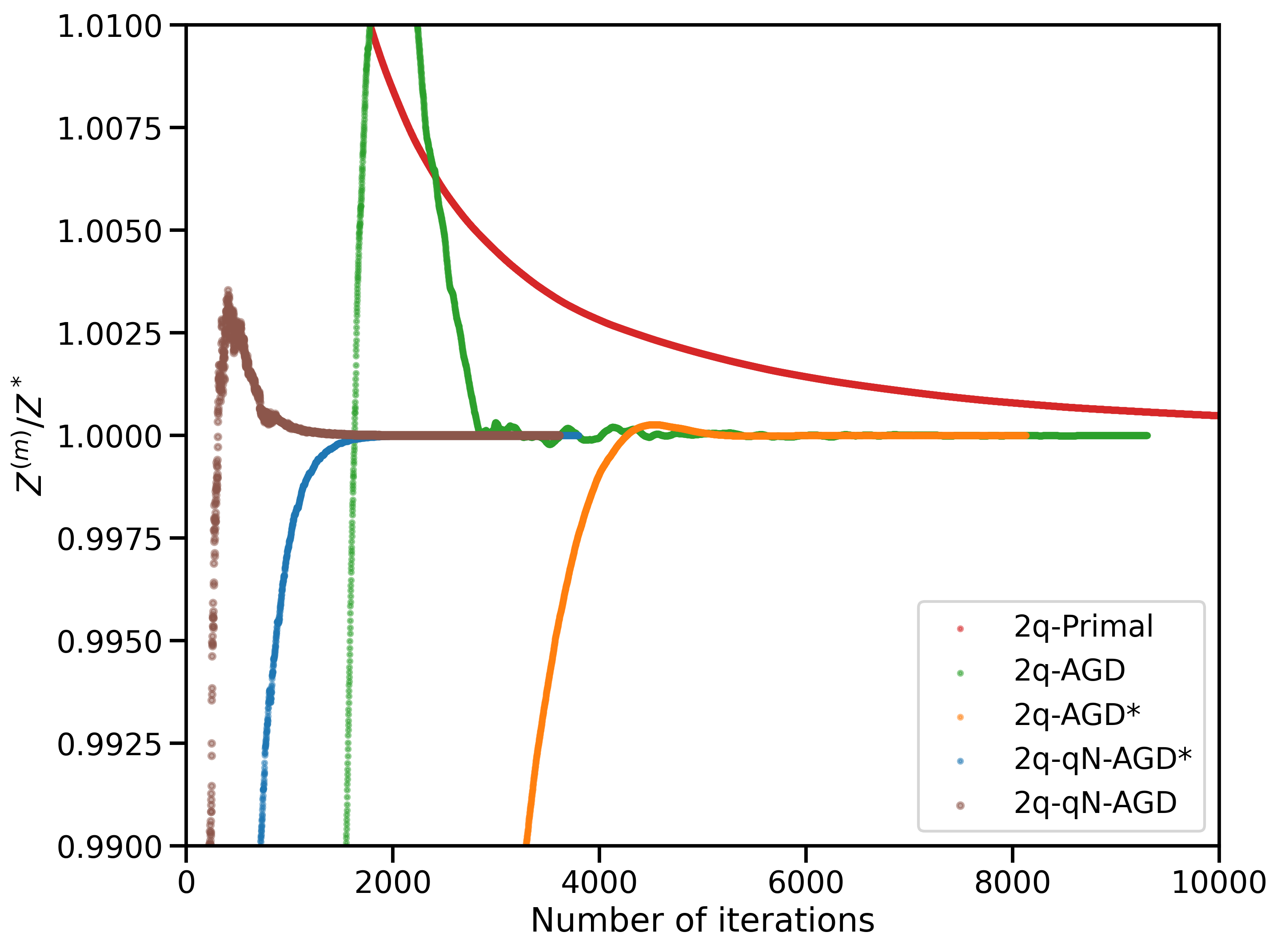}}
    \caption{Convergence performance of the proposed quasi-Newton AGD* (qN-AGD*), AGD*, the original AGD, and the primal TOS algorithm, with three demand levels (1q, 1.5q, 2q). The horizontal axis shows the number of iterations. The vertical axis shows the ratio $Z/Z^*$ between the current objective value $Z$ and the optimal objective value $Z^*$.}
    \label{fig:alg_convergence}
\end{figure}

Figure \ref{fig:alg_convergence} shows the convergence of the objective to the true value for the four dual algorithms and (in addition)the primal TOS algorithm, under three demand levels.
The primal TOS algorithm is our point of reference for the four dual algorithms (qN-AGD*, qN-AGD, AGD*, and AGD). In this experiment, all four dual algorithms converge to the optimal objective value, while the TOS has not converged after 10,000 iterations. Comparing the four dual algorithms, the qN-AGD* is the fastest to reach the optimal objective value, with no oscillation observed in its convergence trajectory under any demand level. Similarly, qN-AGD is faster than first-order AGD and AGD*, but exhibits oscillation as in the case of AGD. These results are consistent with the finding in the literature that the AGD* algorithm reduces the oscillation behavior of the original AGD.

\subsection{Runtime performance with different network sizes}
Given the superior performance of the dual algorithms compared to the primal algorithm, we focus on examining the performance of the dual algorithms for networks with sizes ranging from small to large in this subsection. We consider the default demand matrix (1q) for all the test networks, and the results are summarized in Table~\ref{tab:dual_runtime}.

\begin{table}[H]
\caption{Dual algorithm runtime performances}
\label{tab:dual_runtime}
\centering
\setlength\tabcolsep{8pt}
\begin{tabular}{ccrrrr}
  \toprule
   \multirow{2}{*}{\textbf{Network}} & \textbf{Problem size} & \multicolumn{4}{l}{\textbf{Runtime [s]}}  \\
    & $(|\mathcal{N}|\times |\mathcal{W}|)$ & qN-AGD* & qN-AGD &AGD* & AGD  \\
   \midrule
   Sioux Falls & $1.27\times 10^4$ & 0.25 & 0.43 & 1.95 & 5.78\\
   Berlin-Friedrichshain & $1.13\times 10^5$ & 2.17 & 7.23 &42.83 & 144.12 \\
   Berlin-Tiergarten & $2.31\times 10^5$  & 3.23& 7.70 &142.85 & 410.33 \\
   Anaheim & $5.85\times 10^5$ & 0.58 & 0.65 & 16.87 & 26.71\\
   \multirow{2}{*}{{\makecell{Berlin-Mitte-Prenzlauerberg\\-Friedrichshain-Center}}} &  $9.26\times 10^6$ & 68.99 & 72.01 &1,487.23 & 3,742.31 \\
   & & & & \\
   Chicago-Sketch & $8.69\times 10^7$ & 94.90 & 125.77 & 7,196.18& 9,322.30 \\
\bottomrule
\end{tabular}
\end{table}
We find that our qN-AGD* algorithm converges within a few seconds with the small to medium sized networks. For large networks like Berlin central areas and Chicago sketch we find also very satisfactory runtimes (within 100 seconds). Meanwhile, the runtime performance of qN-AGD is comparable to that of qN-AGD*, but slower. In contrast, AGD* and AGD are at least 20 times slower than qN-AGD* and qN-AGD in medium to large networks. Furthermore, in general, the runtime of AGD* and AGD increases superlinearly with respect to problem size. We note that the runtime for the Anaheim network is shorter in all four algorithms, compared to other networks with similar size. This could be related to the congestion level in the default network setups.

\subsubsection{Runtime performance at different network sizes and demand levels}

In this subsection, we systematically evaluate the computational efficiency of the proposed qN-AGD* algorithm and the qN-AGD (which exhibits similar performance) in real-size networks at different demand levels. We here do not consider AGD* and AGD as they clearly showed inferior convergence.

As shown in Figure \ref{fig:grid_net}, we construct the grid test networks as proposed in \citet{oyama2022markovian} by joining blocks of grids, to exemplify the effects of network size and demand levels on runtime performance. We assume the BPR function $t_{ij} = t_{ij, 0} \left(1 + 0.15 (\frac{x_{ij}}{\kappa_{ij}})^4 \right)$ with free-flow travel time $t_{ij, 0} = 1$, link capacity $\kappa_{ij} = \text{5,000}$, and link length $l_{ij} = 1$, for all links $(i,j) \in \mathcal{E}$. For each origin $o$, the demand for each destination $d \in \mathcal{D}$ is assumed to follow the gravity model:
\begin{align}
    q_o^d = q \cdot \frac{\exp(t_{od, 0})}{\sum_{d \in \mathcal{D}\setminus o} \exp(t_{od, 0})}, \forall o \in \mathcal{O}, \nonumber
\end{align}
where $t_{od, 0}$ is the shortest free-flow travel time between OD pair $od$, and $q$ is the total trips generated from each origin.

\begin{figure}[H]
    \centering
    {\includegraphics[width=0.95\linewidth]{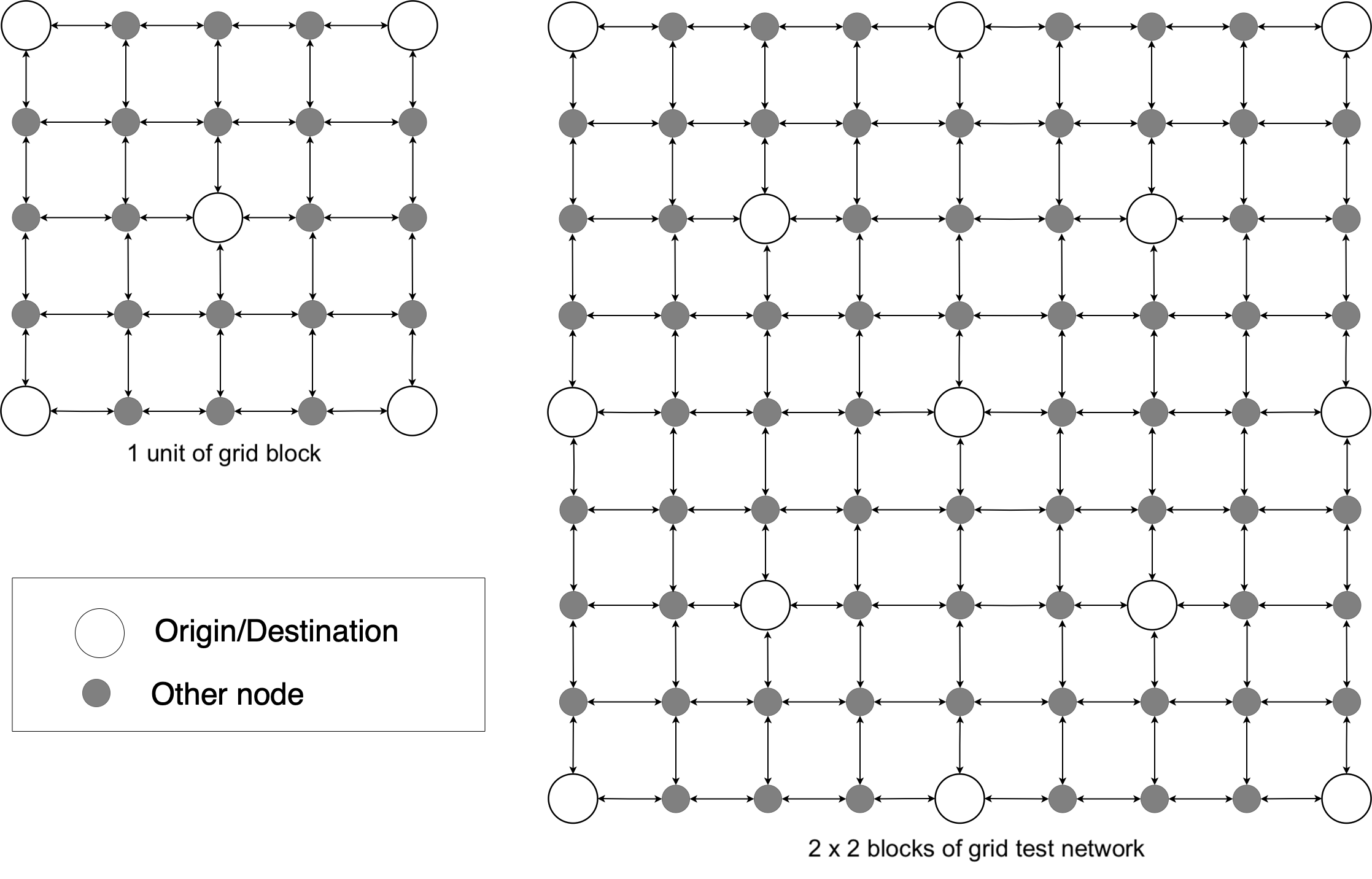}}
    \caption{Bidirectional $k \times k$ blocks of grid test network.}
    \label{fig:grid_net}    
\end{figure}

An indicator for the size of the assignment problem is the number of decision variables, i.e., the number of nodes times the number of traveler types. As shown in Figure~\ref{fig:alg_runtime}, we find the runtime depends about linearly on the problem size. This suggests that our proposed algorithm will scale well to larger networks. Furthermore, the runtime only increases slightly with increasing demands. Note that the grid test networks are congested under our setting. This result further suggests that our proposed algorithm is suitable also for congested networks.

\begin{figure}[H]
    \centering
    \includegraphics[width=0.75\textwidth]{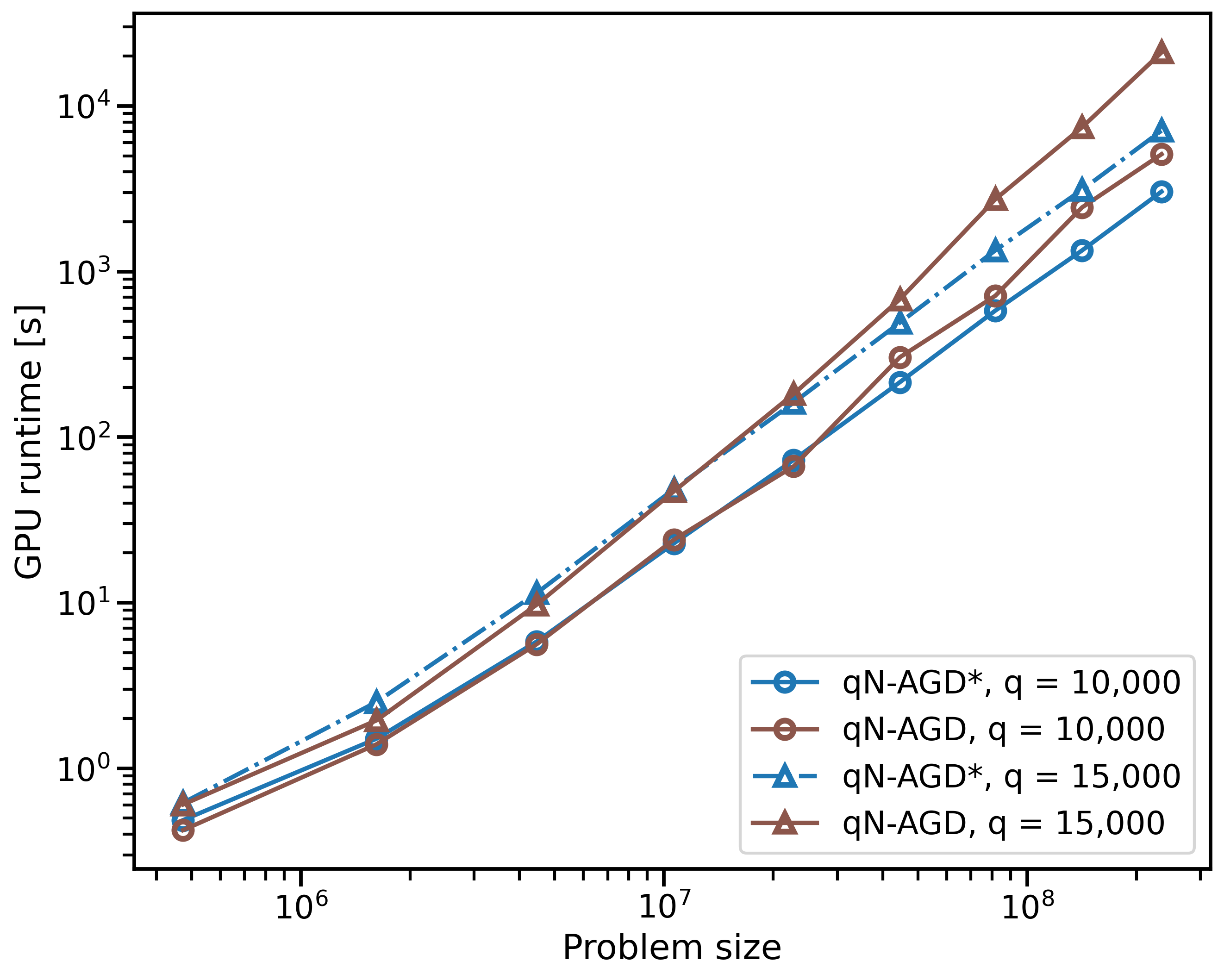}
    \caption{Runtime performance to problem size ($k$ ranging from 4 to 12)}
    \label{fig:alg_runtime}
\end{figure}

\section{Conclusion}\label{sec:conclusion}
This paper has proposed the perturbed utility stochastic traffic assignment model and an accompanying accelerated gradient-based algorithm based on the equivalent dual formulation of the traffic assignment problem. The dual assignment problem is unconstrained with closed-form stochastic network loading, which helps to make our proposed qN-AGD* algorithm very fast.

Our simulation evidence suggests that our proposed algorithm will scale well to larger problems. This is important for making the perturbed utility route choice model competitive with other route choice models for large-scale applications. This is also valuable, as the perturbed utility route choice model has several significant advantages over other route choice models. In particular, it does not require any choice set generation but simply uses the complete network as it is. It generates realistic substitution patterns directly from the network structure. Most of the network is inactive for any origin-destination pair. Not least, it is fast and easy to estimate using linear regression.

%

%

\newpage
\bibliography{MF,mybib}

\end{document}